\documentclass[11pt]{amsart}  
 \usepackage[dvips]{epsfig}
\usepackage{graphicx}
\usepackage{latexsym}
\usepackage{amsfonts,amsmath,amssymb}
\newtheorem{theorem}{Theorem}[section]

\newtheorem{lemma}[theorem]{Lemma}

\newtheorem{corollary}[theorem]{Corollary}
\newtheorem{example}[theorem]{Example}

\author[M. B\'ona]{Mikl\'os  B\'ona}
\title[A new upper bound for 1324-avoiding permutations]{A new upper bound for
1324-avoiding permutations}
\address{\rm M. B\'ona, Department of Mathematics, 
University of Florida,
358 Little Hall, 
PO Box 118105, 
Gainesville, FL 32611--8105 (USA)
}

\date{}

\begin{document}

\begin{abstract} We prove that the number of 1324-avoiding permutations
of length $n$ is less than $(7+4\sqrt{3})^n$.
\end{abstract}

\maketitle

\section{Introduction}
\subsection{Definitions and Open Questions}
 The theory of pattern avoiding permutations has
seen tremendous progress during the last two decades. The key definition is
the following. Let $k\leq n$, let $p=p_1p_2\cdots p_n$ be a permutation of
 length $n$, and
let $q=q_1q_2\cdots q_k$ be a permutation of length $k$. We say that $p$
 avoids $q$ if there 
are no $k$ indices $i_1<i_2<\cdots <i_k$ so that for all $a$ and $b$, 
the inequality $p_{i_a}<p_{i_b}$ holds if and only if the inequality $q_a<q_b$
holds. For instance, $p=2537164$ avoids $q=1234$ because $p$ does not contain
an increasing subsequence of length four. See \cite{combperm} for an overview
of the main results on pattern avoiding permutations.

The shortest pattern for which even some of the most basic questions are open
is $q=1324$, a pattern that has been studied for at least 17 years.
 For instance,
 there is no known exact formula for the number 
$S_n(1324)$ of permutations of length $n$ (or, in what follows,
 $n$-permutations)
 avoiding 1324. Even the 
value of $L(1324)=\lim_{n\rightarrow \infty} \sqrt[n]{S_n(1324)}$ is unknown, though
the limit is known to exist \cite{arratia}.

The best known upper bound for the numbers $S_n(1324)$ was given in 2011
 by Claesson,
Jelinek and Steingr\'{\i}msson \cite{claesson} who proved that for all positive 
integers $n$, the inequality $S_n(1324)<16^n$ holds. The best known lower
bound, $S_n(1324)\geq 9.42^n$, was given by five authors in \cite{albert} in
2005. 

In this paper, we prove the inequality $S_n(1324)<(7+4\sqrt{3})^n$. The proof
introduces a refined version  of a decomposition of 1324-avoiding permutations
given in \cite{claesson}, encodes such permutations by two words over a
4-element alphabet, and then enumerates those words. 

\subsection{Preliminaries} \label{prelim}
In this section, we present a few simple facts that are well-known among
 researchers
working in the area that will be necessary in order to understand some of our
proofs in the subsequent sections. Readers familiar with the area may skip this
section. Proofs that are not given here can be found in \cite{combperm}.

\begin{theorem} Let $q$ be any pattern of length three. Then $S_n(q)=C_n=
{2n\choose n}/(n+1)$, the $n$th Catalan number. In particular, 
$S_n(q)<4^n$. 
\end{theorem}

An entry of a permutation is called a {\em left-to-right} minimum if it is 
smaller than all entries on its left. {\em Right-to-left} maxima are defined
analogously. For instance, in $p=351624$, the left-to-right minima are 3 and 1,
while the right-to-left maxima are 6 and 4. A 132-avoiding permutation is
 completely determined by the set of
its left-to-right minima, and the set of indices that belong to entries that
are left-to-right minima. Indeed, left-to-right minima must always be in
decreasing order. Furthermore, once the set and position of the  left-to-right
 minima 
are given,  the order of elements that are not left-to-right
minima is uniquely determined. To see this, fill the positions that belong to
entries that are not left-to-right minima one by one, going left to right.
In each step, the smallest remaining entry that is larger than the closest
left-to-right minimum $m$ on the right of the position at hand must be placed.
If we do not follow this procedure and place the entry $y$ instead of the
 smaller entry $x$, then the 132-pattern $myx$ is formed. For example,
to find the unique permutation of length 6 whose left-to-right minima are
the entries 1, 3, and 4, and that has left-to-right minimia in the first, second
and fifth position, write the left-to-right minima in the specified positions
in decreasing order, to get $43**1*$, where the $*$ denote positions that are
still empty. Then fill the empty slots with the remaining entries, always
 placing 
the smallest entry that is larger than the closest left-to-right minimum on the
left. In this case, that means first placing 5, then 6, then 2, to get
$435612$. 

In an analogous way, each 213-avoiding permutation is determined by the 
set of its right-to-left maxima, and the set of indices that belong to 
right-to-left maxima.

In preparation to our main results, we  reformulate the facts discussed
in the last 
 paragraphs. Permutations $p=p_1p_2\cdots p_n$ of length $n$ that avoid 132
can be {\em injectively} encoded by ordered pairs of words $(u(p), v(p))$
of length $n$ defined as follows. The $i$th letter of $u(p)$ is 0 if $p_i$ is 
a left-to-right minimum in $p$, and 1 otherwise. The $i$th letter of $v(p)$ is
0 if the entry 
$i$ is a left-to-right minimum in $p$, and 1 otherwise.  The encoding
of 213-avoiding permutations is analogous. 

\section{Coloring entries} \label{coloring}

The starting point of our proof is the following decomposition of 1324-avoiding
permutations, given in \cite{claesson}. 

Let $p=p_1p_2\cdots p_n$ be a 1324-avoiding permutation, and let us color each 
entry of $p$ red or blue as we move from left to right, 
according the following rules. 
\begin{enumerate}
\item If coloring $p_i$ red would create a 132-pattern with all red entries,
then color $p_i$ blue, and
\item if there already is a blue entry smaller than $p_i$, then color 
$p_i$ blue;
\item otherwise color $p_i$ red.  
\end{enumerate}

It is then proved in \cite{claesson} that the red entries form a 132-avoiding
permutation and
the blue entries form a 213-avoiding permutation. From this, it is not difficult
to prove that the number of 1324-avoiding $n$-permutations is less than $16^n$. 
Indeed, there are at most $2^n$ possibilities for the set of the red entries
(the 
blue entries being the remaining entries), and there are at most $2^n$
 possibilities for the positions in
 which
red entries are placed (the blue entries then must be placed in the remaining
positions). Once the set and positions of the $k$ red entries are known,
 there are $C_k<4^k$ possibilities for their permutation, just as there are
 $C_{n-k}<
4^{n-k}$ possibilities for the permutation of the blue entries, completing the
proof of the inequality $S_n(1324)<16^n$.

\section{Refining the coloring} \label{refined}
In this section, we improve the upper bound on $S_n(1324)$ by using 
a more refined
decomposition of 1324-avoiding permutations, which enables
 us to carry out a more careful counting argument. 
Let us color each entry of the 1324-avoiding permutation 
$p=p_1p_2\cdots p_n$ red or blue as in
Section \ref{coloring}. 
Furthermore, let us mark each entry of $p$ with one of the letters $A$,
 $B$, $C$, or

$D$ as follows.

\begin{enumerate}
\item Mark each red entry that is a left-to-right minium in the partial 
permutation of red entries by $A$, 
\item mark each red entry that is not a left-to-right minimum in the partial
permutation of red entries by $B$, 
\item mark each blue entry that is not a right-to-left maximum in the partial
 permutation
of blue entries by $C$, and
\item mark each blue entry that is a right-to-left maximum in the partial
 permutation of
blue entries by $D$.
\end{enumerate}

Call entries marked by the letter $X$ entries of {\em type} $X$.
Let $w(p)$ be the $n$-letter word over the alphabet $\{A,B,C,D\}$ defined above.
In other words, the $i$th letter of $w(p)$ is the type of $p_i$ in $p$. 
Let $z(p)$ be the $n$-letter word over the alphabet $\{A,B,C,D\}$ whose $i$th
letter is the type of the entry $i$ in $p$. 

\begin{example} Let $p=3612745$. Then the subsequence of red entries of $p$
is $36127$, the subsequence of blue entries of $p$ is $45$, so 
$w(p)=ABABBCD$, while $z(p)=ABACDBB$. 
\end{example}

The following lemma shows a property of $w(p)$ that will enable us to
improve the upper bound on $S_n(1324)$. Let us say that a word $w$ has
a  {\em $CB$-factor} if somewhere in $w$, a letter $C$ is immediately followed
by a letter $B$. 

\begin{lemma} \label{noCB} If $p$ is 1324-avoiding, then $w(p)$ has no 
$CB$-factor.
\end{lemma}

\begin{proof}
Let us assume that $C_1$ is the $i$th letter of $w(p)$, and $B_1$ is
 the $(i+1)$st letter of $w(p)$. 
That means that $p_i>p_{i+1}$, otherwise the fact that $p_i$ is blue would
 force $p_{i+1}$ to be blue.
Furthermore, since $p_i$ is not a right-to-left maximum, there is an entry
 $d$ on the right of $p_i$
(and on the right of  $p_{i+1}$) so that $p_i<d$. Similarly, since $p_{i+1}$
 is not
 a left-to-right minimum, 
there is an entry $a$ on its left so that $a<p_{i+1}$. However, then
 $ap_{i}p_{i+1}d$ is a 1324-pattern, which is a contradiction.
\end{proof}

\begin{lemma}
If $p$ is 1324-avoiding, then there is no entry $i$ in $p$ so that $i$
 is of type  $C$
and $i+1$ is of type $B$. 
\end{lemma}

\begin{proof} Analogous to the proof of lemma \ref{noCB}. If such a pair
existed, $i$ would have to be on the right of $i+1$, since $i$ is blue
and $i+1$ is red. As $i$ is not a right-to-left maximum, there would be
a larger entry $d$ on its right. As $i+1$ is not a left-to-right minimum, 
there would be a smaller entry $a$ on its left. However, then $a(i+1)id$
would be a 1324-pattern.
\end{proof}

\begin{lemma} Let $h_n$ be the number of words of length $n$ that consist
of letters $A$, $B$, $C$ and $D$ that have no $CB$-factors. 
Then we have 
\[H(x)=\sum_{n\geq 0}h_nx^n = \frac{1}{1-4x+x^2}.
\] This implies
\begin{equation}
\label{exact} h_n=\frac{3+2\sqrt{3}}{6} \cdot \left(2+\sqrt{3}\right)^n
+\frac{3-2\sqrt{3}}{6} \cdot \left(2-\sqrt{3}\right)^n .\end{equation}
\end{lemma}

\begin{proof}
We claim that if $n\geq 2$, then $h_n=4h_{n-1}-h_{n-2}$. Indeed, take any
of the $h_{n-1}$ words of length $n-1$ that have the given property. 
Affix any of the four letters of the alphabet to the end of each such word.
The result is a word counted by $h_n$, except in the $h_{n-2}$ cases in which
the last two letters are $C$ and $B$, in that order.

Together with the initial conditions $h_0=1$ and $h_1=4$, this leads to the
functional equation
\[H(x)-4x-1= 4x(H(x)-1)-x^2H(x).\]
Expressing $H(x)$, we obtain
\[H(x)=\frac{1}{1-4x+x^2}\] as claimed.
 It is now routine to find the exact formula
for $h_n$ using partial fractions.
\end{proof}

The following, simple but crucial lemma tells us that the ordered
pair $(w(p),z(p))$ completely determines the 1324-avoiding 
permutation $p$.
 Readers who prefer
may consult Section \ref{prelim} first for some background on this argument.
 
\begin{lemma} \label{determines}
Let $Av_n(1324)$ be the set of all 1324-avoiding $n$-permutations. Then
the map $f:Av_n(1324)\rightarrow H_n\times H_n$, given by $f(p)=(w(p),z(p))$
 is injective.
\end{lemma}

\begin{proof} Let $H_n$ be the set of all words of length $n$ over the
alphabet $\{A,B,C,D\}$ in which a letter $C$ is never immediately followed
by a letter $B$. Then $|H_n|=h_n$. Let $(w,z)\in H_n$, and let us assume
that $f(p)=(w,z)$, that is, that $w(p)=w$, and $z(p)=z$ for some 
$p\in Av_n(1324)$. 

Then $w$ tells us for which indices $i$ the entry $p_i$ will be of type $A$, 
namely for the indices $i$ for which the $i$th letter of $w$ is $A$.
Similary, $w$ tells us the indices $j$ for which the entry $p_i$ is of type
$B$, type $C$, or type $D$. 

After this, we can use $z$ to figure out which {\em entries} of $p$ are 
of type $A$, type $B$, type $C$ or type $D$. 

There remains to show that this information completely determines $p$, that is,
that there is at most one permutation that avoids 1324 and satisfies all the
type requirements imposed by $w$ and $z$. 

In order to see this, note that entries of type $A$ must be in decreasing order
 in their positions. Entries
 of type $D$
must be in decreasing order in their positions. Once these entries are placed, 
entries of type $B$ must be placed in their positions from left to right, so
 that in 
each step, the smallest available entry is placed that is larger than the
 closest
entry of type $A$ on the left. (Otherwise a red 132-pattern is formed.)
Similarly, the entries of type $D$ must be placed in their positions from
 the right, 
so that in each step, the largest available entry is placed that is smaller
 than the 
closest right-to-left maximum on the right. (Otherwise a blue 213-pattern
 is formed.)
\end{proof}

\begin{corollary} \label{upperbound}
For all positive integers $n$, the inequality 
\[S_n(1324) < h_{n-1}^2 \] holds.
\end{corollary}

\begin{proof} The fact that $S_n(1324)<h_n^2$ is immediate from the injective
property of $f$ that we have just proved in Lemma \ref{determines}. In order
to complete the proof of this Corollary, note that the image of $f$ consists
of ordered pairs $(w(p),z(p))$ in which both $w(p)$ and $z(p)$ starts with 
an $A$, since both $p_1$ and 1 are always red, and left-to-right minima within
the string of red entries (and even in all of $p$).
\end{proof}

\begin{center} {\bf Acknowledgment} \end{center}
I am grateful to Vincent Vatter for some valuable remarks 
streamlining the proofs of two of my lemmas.

\end{document}